\documentclass{amsart}
\usepackage[a4paper]{geometry}
\usepackage{amsfonts}
\usepackage{pdfsync}
\usepackage{amsmath}
\usepackage{amssymb}
\usepackage{amsthm}
\usepackage{mathrsfs}
\usepackage{color}
\usepackage{enumerate}

\newtheorem{theorem}{Theorem}[section]
\newtheorem{lemma}[theorem]{Lemma}
\newtheorem{corollary}[theorem]{Corollary}
\newtheorem{proposition}[theorem]{Proposition}

\newtheorem{Conjecture}[theorem]{Conjecture}
\theoremstyle{definition}
\newtheorem{remark}[theorem]{Remark}

\newcommand{\dd}{\mathrm{d}}
\newcommand{\R}{{\mathord{\mathbb R}}}

\newcommand{\N}{{\mathord{\mathbb N}}}

\newcommand{\dom}[1]{\mathrm{Dom}(#1)}

\newcommand{\ie}{\emph{i.e.}}

\newcommand{\cf}{\emph{cf.}}
\newcommand{\sii}{L^2}
\newcommand{\eps}{\varepsilon}
\newcommand{\diag}{\mathop{\mathrm{diag}}\nolimits}
\newcommand{\domain}{\Pi}

\title{Location of hot spots in thin curved strips}
\author{David Krej\v{c}i\v{r}\'{i}k}
\address{Department of Mathematics\\ Faculty of Nuclear Sciences and Physical Engineering \\ Czech Technical University in Prague\\ Trojanova 13\\ 120 00 Prague\\ Czechia}
\email{david.krejcirik@fjfi.cvut.cz}

\author{Mat\v{e}j Tu\v{s}ek}
\address{Department of Mathematics\\ Faculty of Nuclear Sciences and Physical Engineering \\ Czech Technical University in Prague\\ Trojanova 13\\ 120 00 Prague\\ Czechia}
\email{matej.tusek@fjfi.cvut.cz}

\date{August 29, 2018 (Prague)}

\begin{document}

\begin{abstract}
The maxima and minima of Neumann eigenfunctions 
of thin tubular neighbourhoods of curves on surfaces
are located in terms of the maxima and minima of Neumann eigenfunctions 
of the underlying curves.
In particular, the hot spots conjecture 
for a new large class of domains 
(possibly non-convex and non-Euclidean)
is proved.
\end{abstract}

\maketitle

%---------------------%
\section{Introduction}
%---------------------%

\subsection{Motivation I: the hot spots conjecture}
Consider the Neumann eigenvalue problem for the Laplacian
in a bounded domain~$\Omega$ of
a two-dimensional oriented Riemannian manifold:
\begin{equation}
\left\{
\begin{aligned}
  -\Delta u &= \lambda u 
  && \mbox{in} \ \Omega \,,
  \\
  \frac{\partial u}{\partial \nu} &= 0
  && \mbox{on} \ \partial\Omega \,,
\end{aligned}
\right.
\end{equation}
where~$\nu$ denotes the outward unit normal vector field 
on the boundary~$\partial\Omega$.  
Assuming that~$\Omega$ is sufficiently regular,
it is well known that the eigenvalues can be arranged
in an increasing sequence $\{\lambda_n\}_{n=1}^\infty$, 
where each eigenvalue is repeated according to its multiplicity,
so that
$$
  0=\lambda_1 < \lambda_2 \leq \lambda_3 \leq \dots \to \infty
  \,.
$$
The corresponding eigenfunctions~$\{u_n\}_{n=0}^\infty$ 
can be chosen in such a way that they are real-valued
and form a complete orthonormal set in $\sii(\Omega)$.

Since the first eigenfunction~$u_1$ is constant in~$\Omega$,
all the other eigenfunctions must change sign. Therefore, their global maxima (respectively, minima) in $\overline{\Omega}$ are always positive (respectively, negative).
Note that the roles of maxima and minima are interchanged 
after multiplying an eigenfunction by minus one 
and that higher eigenvalues may be degenerate.
One of the most challenging open problems in spectral geometry 
is about the location of the maxima and minima 
of the \emph{second} eigenfunction.
The celebrated conjecture of Rauch's from 1974
states that any eigenfunction corresponding to~$\lambda_2$ 
attains its maximum and minimum at boundary points only:

\begin{Conjecture}[Hot spots conjecture]
\label{Conj.Rauch}
$
\displaystyle
  \forall x \in \Omega, \quad
  \min_{\partial\Omega} u_2 < u_2(x) < \max_{\partial\Omega} u_2
$.
\end{Conjecture}

According to Ba\~nuelos and Burdzy~\cite{Banuelos-Burdzy_1999}, 
Conjecture~\ref{Conj.Rauch} was indeed raised by Rauch 
during a conference in 1974 but 
``it has never appeared in print under his name''.
In fact, the conjecture is usually stated for Euclidean domains,
\ie\ $\Omega\subset\R^2$, but the present extension 
to manifolds is obvious. 
It is certainly true that Conjecture~\ref{Conj.Rauch}
does not appear in Rauch's conference report~\cite{Rauch_1975},
where however a strong heuristic support 
for the validity of it can be deduced from: 
Since the first eigenvalue is zero, the large-time behaviour
of the heat semigroup generated by the Neumann Laplacian
is determined by the second eigenfunction 
and Neumann boundary conditions model an insulating interface,
so it is expected that ``hot spots'' and ``cold spots'' 
of a medium living inside~$\Omega$ will move towards 
the boundary~$\partial\Omega$ for large times.

Conjecture~\ref{Conj.Rauch} 
has been proved for various classes of domains, 
typically requiring special forms 
\cite[Sec.~II.5]{Kawohl_1985},
certain symmetry 
\cite{Banuelos-Burdzy_1999,Jerison-Nadirashvili_2000,Pascu_2002}
or smallness of a parameter 
\cite{Atar-Burdzy_2004,Miyamoto_2009,Miyamoto_2013,Siudeja_2015}.
The most general positive result obtained so far 
was given by Atar and Burdzy in 2004 \cite{Atar-Burdzy_2004},
who showed that the conjecture holds for all Euclidean Lipschitz domains 
delimited by graphs of two Lipschitz functions 
with Lipschitz constant equal to one. 

On the negative side,
there exist counterexamples for multiply connected Euclidean domains
\cite{Burdzy-Werner_1999,Bass-Burdzy_2000,Burdzy_2005}
and domains on manifolds \cite{Freitas_2002}.
The current state of the art is that the conjecture is expected
to hold for convex or even simply connected Euclidean domains
\cite[Conj.~1.2]{Burdzy_2005}.

The primary objective of this paper is to provide 
a new large class of domains
-- non-convex and possibly on manifolds --
for which the hot spots conjecture holds.
These domains are \emph{strips}, that is, 
tubular neighbourhoods of curves 
on an arbitrary two-dimensional Riemannian manifold.
We prove the conjecture under the additional assumption
that the radius of the tubular neighbourhood is small,
so our domains are generally covered neither by
the class considered in~\cite{Atar-Burdzy_2004}
nor the other works quoted above.
Furthermore, the technique we use enables us to locate 
the maxima and minima of \emph{all} the Neumann eigenfunctions 
in the thin strips.

\subsection{Strips on surfaces: location of hot spots}
To precisely state our results in the context of the hot spots conjecture,
we need to begin with geometric preliminaries.
Let~$\mathcal{A}$ be a two-dimensional oriented Riemannian manifold 
(not necessarily embedded in~$\R^3$)
and denote by~$K$ its Gauss curvature.
Since our results are established under regularity assumptions
which are not expected to be optimal, let us assume from 
the beginning that~$\mathcal{A}$ is smooth, \ie~$C^\infty$.
Given a positive number~$L$,
let $\Gamma:[0,L]\to\mathcal{A}$ be a smooth embedded curve,
which is assumed to be parametrised by its arc-length.
Denoting by $T:=\Gamma'$ the unit tangent vector field of~$\Gamma$,
we introduce a unit normal vector field~$N$ of~$\Gamma$
by the requirements that~$N$ is smooth, of norm~$1$, orthogonal to~$T$
and that the couple $\{T(s),N(s)\}$ 
is positively oriented for all $s \in [0,L]$. 
The curvature~$\kappa$ of~$\Gamma$ is then defined 
by the Frenet formula $T'=\kappa N$
(it is the \emph{geodesic} curvature of~$\Gamma$
if~$\mathcal{A}$ is embedded in~$\R^3$).

Given a positive number~$\eps$,
we define a \emph{strip}~$\Omega_\eps$ \emph{of width}~$2\eps$
to be the set of all the points~$x$ in~$\mathcal{A}$
for which there exists a geodesic of length less than~$\eps$ 
from~$x$ meeting~$\Gamma$ orthogonally. In other words,
\begin{equation}\label{strip}
  \Omega_\eps := \mathscr{L}_\eps(\domain)
  \quad \mbox{with} \quad
  \mathscr{L}_\eps(s,t) := \exp_{\Gamma(s)}\big(\eps \, t\, N(s)\big) \,,
  \quad 
  \domain := (0,L) \times (-1,1)
  \,,
\end{equation}
where~$\exp_x$ is the exponential map of~$\mathcal{A}$ 
at $x \in \mathcal{A}$.
Note that $s \mapsto \mathscr{L}_\eps(s,t)$
traces the curves parallel to~$\Gamma$
at a fixed distance~$\eps |t|$,
while the curve $t \mapsto \mathscr{L}_\eps(s,t)$
is a geodesic orthogonal to~$\Gamma$
for any fixed~$s$.
We also note that in the Euclidean case, 
\ie~$\mathcal{A}=\R^2$,
we have $K=0$ and
$
  \mathscr{L}_\eps(s,t)
  = \Gamma(s) + N(s) \, \eps \, t  
$.

Let $\{\lambda_n^\eps\}_{n=1}^\infty$ denote 
the increasing sequence of eigenvalues
of the Neumann Laplacian in the strip~$\Omega_\eps$, 
where each eigenvalue is repeated according to its multiplicity,
and let $\{u_n^\eps\}_{n=1}^\infty$ denote the set of corresponding
real-valued eigenfunctions.
At the same time, 
let $\{\lambda_n^0\}_{n=1}^\infty$ denote the increasing sequence of eigenvalues
of the Neumann Laplacian on the interval~$(0,L)$, 
where each eigenvalue is repeated according to its multiplicity,
and let $\{u_n^0\}_{n=1}^\infty$ denote the set of corresponding
real-valued eigenfunctions.
Since the spectral problem for the latter 
can be solved explicitly in terms of sines and cosines,
it is well known that all the eigenvalues~$\lambda_n^0$ are simple
and that each eigenfunction~$u_n^0$ with $n\geq 2$ 
admits~$n$ stationary points 
\begin{equation} \label{eq:stac_points}
  \Big\{
  s_{m}^{(n)} := \frac{m}{n-1} \, L
  \Big\}_{m=0}^{n-1} \subset [0,L]
  \,,
\end{equation}
which correspond to, say, 
$\lfloor n/2 \rfloor$ minima and $\lceil n/2 \rceil$ maxima.
More specifically,
the eigenfunctions can be chosen in such a way that $s_m^{(n)}$
with $m$~even (respectively, $m$~odd) correspond to
maxima (respectively, minima).
The main result	 of this paper says that 
the extremal points of the eigenfunctions of
the Neumann Laplacian in~$\Omega_\eps$ 
are located near the extremal points of the eigenfunctions of
the Neumann Laplacian 
on the curve~$\Gamma$ 
(which can be identified with the interval $[0,L]$). 
For a concise presentation of the result and also with future purposes in mind, 
we firstly introduce the following subsets of~$\overline{\domain}$:
$$
  S_m^{(n)}(\delta)
  :=\big((s_m^{(n)}-\delta,s_m^{(n)}+\delta)
  \cap(0,L)\big)\times(-1,1)
$$ 
for all $m\in\{0,1,\ldots,n-1\}$ and  
\begin{align*}
  \tilde{S}_0^{(n)}
  &:=\{0\}\times[-1,1],
  \\
  \tilde{S}_m^{(n)}(\delta)
  &:=\overline{S_m^{(n)}(\delta)} 
  \quad\text{ for }\; m\in\{1,2,\ldots,n-2\} ,
  \\
  \tilde{S}_{n-1}^{(n)}
  &:=\{L\}\times[-1,1] ,
\end{align*}
where $\delta>0$. 
To make the notation more flexible, 
we will sometimes write  $\tilde{S}_m^{n}$ instead of $\tilde{S}_m^{n}(\delta)$ 
even in the case $m\in\{1,2,\ldots,n-2\}$.
Then our main result can be stated as follows:
\begin{theorem}\label{Thm.location}
For any  $N \in \{2,3,\ldots\}$ 
and $\delta>0$, 
there exists a positive constant~$\eps_0$ 
depending on~$\delta$, $N$, $\Gamma$ and~$\mathcal{A}$ such that, 
for all $\eps < \eps_0$ and any $n \in \{2,\dots,N\}$,
the eigenvalue~$\lambda_n^\eps$ is simple 
and the corresponding eigenfunction~$u_n^\eps$ 
can be chosen in such a way that
\begin{align}
  \big\{ x \in \overline{\Omega}_\eps : \
  \max_{\overline{\Omega}_\eps}u_n^\eps=u_n^\eps(x) 
  \big\}   
  &\subset  \bigcup_{\substack{m=0\\ m\text{ even}}}^{n-1}
  \mathscr{L}_\eps\big(\tilde{S}_m^{(n)}\big) ,	
  \label{eq:location_max}
  \\
  \big\{ x \in \overline{\Omega}_\eps : \ 
  \min_{\overline{\Omega}_\eps}u_n^\eps=u_n^\eps(x) 
  \big\}   
  &\subset  \bigcup_{\substack{m=0\\ m\text{ odd}}}^{n-1}
  \mathscr{L}_\eps\big(\tilde{S}_m^{(n)}\big) .
  \label{eq:location_min}
\end{align}
Moreover, $u_n^\eps$ has no stationary points in~$\Omega_\eps$ 
outside the sets on the right-hand sides of~\eqref{eq:location_max} 
and~\eqref{eq:location_min}; 
and in $\mathscr{L}_\eps\big(S_m^{(n)}(\delta)\big)$,
$m\in\{1,2,\ldots,n-2\}$, $u_n^\eps$ will not attain  
local minima (respectively, maxima) if $m$ is even (respectively, odd).
\end{theorem}

Noticing that the side images  
$\mathscr{L}_\eps\big(\tilde{S}_0^{(n)}\big)$ 
and $\mathscr{L}_\eps\big(\tilde{S}_{n-1}^{(n)}\big)$
are subsets of the boundary~$\partial\Omega_\eps$ for all $n \geq 2$,
the case $N=2$ of the theorem particularly implies
the hot spots conjecture for the thin strips:
\begin{corollary}\label{Corol.location} 
There exists a positive constant~$\eps_0$  
depending on~$\Gamma$ and~$\mathcal{A}$ such that, 
for all $\eps < \eps_0$ and all $x \in \Omega_\eps$,
$$
  \min_{\mathscr{L}_\eps(\tilde{S}_{1}^{(2)})} 
  u_2^\eps
  < u_2^\eps(x) < 
  \max_{\mathscr{L}_\eps(\tilde{S}_{0}^{(2)})} 
  u_2^\eps
  \,.
$$
\end{corollary}

The dependence of the critical half-width~$\eps_0$
on the geometry on~$\Gamma$ and~$\mathcal{A}$ 
is discussed in more detail in 
Remark~\ref{Rem.dependence.bis} below.
In particular, a certain uniformity of the result can be achieved.

\subsection{Motivation II: quantum graphs}
Another motivation of this paper is to provide precise estimates
for the Neumann eigenvalues and eigenfunctions in the thin strips,
which is of current interest in the context of approximation 
of quantum dynamics in a thin nanostructure by 
an effective model in a lower-dimensional manifold.
We refer to a recent monograph by Post~\cite{Post} 
for an account on the topic when the nanostructure is modelled by a graph
and for many references.

For experts in the community it is known that
the Neumann Laplacian in a tubular neighbourhood
of a submanifold of a Riemannian manifold converges in a sense
to the Laplace-Beltrami operator on the submanifold
(see~\cite{Schatzman_1996} for a classical reference 
on the convergence of spectra).
Hence, it is expected that 
the Neumann eigenfunctions~$u_n^\eps$ of~$\Omega_\eps$
converges \emph{in a sense} 
to Neumann eigenfunctions~$u_n^0$ of~$\Gamma$. 
Since the hot spots conjecture is clearly satisfied for the latter,
the main idea behind the proof of Theorem~\ref{Thm.location} 
is to establish the convergence in a sufficiently good topology.

In this paper, as the width of the strip tends to zero,
we first establish the operator convergence 
between the Neumann Laplacians in~$\Omega_\eps$ and~$\Gamma$
in a norm-resolvent sense
(see~\cite{Saito_2001} for a related result in the Euclidean case).
Furthermore, using additional methods, 
we obtain the convergence of eigenfunctions 
together with their derivatives in the \emph{uniform} topology
of sort $C^{2,\gamma}$. This result seems to be new in the context of quantum graphs
and Theorem~\ref{Thm.location} actually follows 
as a consequence of it.
 
\subsection{Ideas of the proof}\label{Sec.ideas} 
Let us now comment on the proof of Theorem~\ref{Thm.location}
in more detail.
Our proof consists roughly of four steps,
combining various methods of differential geometry,
spectral analysis, and theory of elliptic partial differential equations:
\begin{enumerate}[I.]
\item Fermi coordinates;
\hfill (Section~\ref{Sec.geometry})
\item Norm-resolvent convergence;
\hfill (Section~\ref{Sec.nrs})
\item Convergence of eigenfunctions;
\hfill (Section~\ref{Sec.efs})
\item Convergence of extremal points.
\hfill (Section~\ref{Sec.critical})
\end{enumerate}

\subsubsection*{Ad \emph{Step~I}}
First of all, using the natural parametrisation of~$\Omega_\eps$
suggested by the very definition~\eqref{strip}, 
we identify the curved and $\eps$-dependent strip~$\Omega_\eps$ 
with the Riemannian manifold $(\domain,G_\eps)$, 
where $G_\eps := \diag(f_\eps^2,\eps^2)$ 
is an $\eps$-dependent metric.
Here~$f_\eps$ is 
(up to a factor~$\eps$) the Jacobian of the scaled exponential map~$\mathscr{L}_\eps$ 
and as such~$f_\eps$ solves the Jacobi equation, \cf~\cite{Ha_1964} or \cite[Sec.~2]{Gray}, 
\begin{equation}\label{eq:Jacobi}
\left\{
\begin{alignedat}{2}
  \partial_t^2 f_\eps(s,t) 
  +\eps^2 \, K_\eps(s,t) \, f_\eps(s,t) &=0 \,,
  \\
  f_\eps(s,0)&=1 \,, 
  \\
  \partial_t f_\eps(s,0)&=-\eps\,\kappa(s)
  \,,
\end{alignedat}
\right.
\end{equation}
for every $t \in [-1,1]$ taking $s \in [0,L]$ as a fixed parameter,
where
\begin{equation} \label{eq:K_scaling}
K_\eps := K \circ \mathscr{L}_\eps
\end{equation}
is the Gauss curvature~$K$
expressed in the Fermi ``coordinates'' $(s,t)\in\domain$.
In this way, we may identify 
the Neumann Laplacian in $\sii(\Omega_\eps)$
with the Laplace-Beltrami operator  
\begin{equation}\label{operator}
  H_\eps
  :=-f_\eps^{-1}\partial_s f_\eps^{-1}\partial_s
  -\eps^{-2}f_\eps^{-1}\partial_t f_\eps\partial_t
  \qquad \mbox{in} \qquad
  \sii\big(\domain, f_\eps(s,t) \, \dd s \, \dd t \big)
  ,
\end{equation}
subject to Neumann boundary conditions.
As usual, $H_\eps$~is understood as the self-adjoint 
operator in the Hilbert space
$\sii\big(\domain, f_\eps(s,t) \, \dd s \, \dd t \big)$
associated with the closed quadratic form
\begin{equation} \label{eq:form_eps}
  h_\eps[\psi] := \int_\domain f_\eps^{-1} \, |\partial_s\psi|^2 
  + \eps^{-2} \int_\domain f_\eps \, |\partial_t\psi|^2
  \,, \qquad
  \psi \in \dom{h_\eps}  := W^{1,2}(\domain)
  \,.
\end{equation}

\subsubsection*{Ad \emph{Step~II}}
Since $f_\eps \to 1$ as $\eps \to 0$ uniformly in~$\domain$,
it is expected that~$H_\eps$ will be approximated by 
the operator 
\begin{equation*}
  H_0 := -\partial_{s}^{2} - \eps^{-2} \, \partial_{t}^{2}
  \qquad \mbox{in} \qquad
  \sii(\domain)
  \,,
\end{equation*}
subject to Neumann boundary conditions.
Again, $H_0$~is understood as the self-adjoint 
operator in the Hilbert space $\sii(\domain)$
associated with the closed quadratic form
\begin{equation} \label{eq:form_0}
  h_0[\psi] := \int_\domain |\partial_s\psi|^2 
  + \eps^{-2} \int_\domain |\partial_t\psi|^2
  \,, \qquad
  \psi \in \dom{h_0}  := W^{1,2}(\domain)
  \,.
\end{equation}
Since~$H_\eps$ and~$H_0$ act on different Hilbert spaces,
it is convenient to introduce the unitary transform  
$$
  U_\eps:
  L^2\big(\domain, f_\eps(s,t) \, \dd s \, \dd t\big)
  \to L^2(\domain) :
  \big\{ \psi\mapsto f_\eps^{1/2}\psi \big\}
  \,.
$$
Then we may justify the approximation of~$H_\eps$ by~$H_0$
via the following resolvent convergence:
\begin{theorem}\label{theo:res_conv}
There exist positive constants~$C$ and~$\eps_0$
depending on~$\Gamma$ and~$\mathcal{A}$ such that, 
for all $\eps<\eps_0$,	
\begin{equation} \label{eq:res_conv}
  \big\|
  U_\eps(H_\eps +1)^{-1}U_{\eps}^{-1}-(H_0+1)^{-1}
  \big\|
  \leq C \eps
  \,. 
\end{equation}
\end{theorem}

In the proof of Theorem~\ref{theo:res_conv},
we use the idea of establishing the norm-resolvent convergence
of operators with help of associated sesquilinear forms
(see~\cite[Thm.~6.3.6]{Kato} for the general criterion
and~\cite[Sec.~3]{K5} for its implementation 
in a geometrically analogous setting).

\subsubsection*{Ad \emph{Step~III}}
As a consequence of Theorem~\ref{theo:res_conv},
we get the convergence of eigenvalues of~$H_\eps$
to the eigenvalues of~$H_0$.
This statement is not precise 
because~$H_0$ is still $\eps$-dependent,
so its eigenvalues depend on~$\eps$.
However, it is easy to see from the decoupled form of~$H_0$ that,
for any $N \in \N :=\{1,2,\ldots\}$, there exists 
a positive number~$\eps_N$ depending on~$L$
such that, for all $\eps \leq \eps_N$ and any $n \in \{1,\dots,N\}$,
the $n$th eigenvalue of~$H_0$ 
is equal to the $n$th eigenvalue of the Neumann Laplacian
in the interval~$(0,L)$.
The latter was denoted by~$\lambda_n^0$  
and the $n$th eigenvalue of~$H_\eps$ 
coincides with~$\lambda_n^\eps$
(the $n$th eigenvalue of the Laplacian in~$\Omega_\eps$).  
Hence, 
\begin{equation}\label{ev-conv}
  | \lambda_n^\eps - \lambda_n^0 | \leq C \eps
  \,,
\end{equation}
where~$C$ is an $\eps$-independent constant
(possibly different from that of Theorem~\ref{theo:res_conv}).
In particular, the $n$th eigenvalue of~$H_\eps$ 
is simple for all sufficiently small~$\eps$
and converges to the $n$th eigenvalue of~$H_0$.  

At the same time, recalling the previously introduced notation~$u_n^0$
for the $n$th eigenfunction
of the Neumann Laplacian in the interval~$(0,L)$,
it is easy to see that, 
for all $\eps \leq \eps_N$ and any $n \in \{1,\dots,N\}$,    
the tensor product $\psi_n^0 := u_n^0 \otimes 1$ 
is the $n$th eigenfunction of~$H_0$.
Let~$\psi_n^\eps$ denote the $n$th eigenfunction of~$H_\eps$
and recall the notation~$u_n^\eps$ for the $n$th eigenfunction
of the Neumann Laplacian in the strip~$\Omega_\eps$. 
Assuming that $\psi_n^\eps \in \sii(\domain)$
and $u_n^\eps \in \sii(\Omega_\eps)$ have the same norm,
we have the relation
$\psi_n^\eps = \sqrt{\eps}\,u_n^\eps \circ \mathscr{L}_\eps$.
It follows from Theorem~\ref{theo:res_conv} that~$\psi_n^\eps$ 
can be normalised in such a way that
\begin{equation}\label{L2-conv} 
  \| \psi_n^\eps - \psi_n^0 \|_{\sii(\domain)} \leq C\eps
  \,,
\end{equation}
where~$C$ is an $\eps$-independent constant
(possibly different from that of Theorem~\ref{theo:res_conv}).

Unfortunately, the convergence~\eqref{L2-conv} is too weak 
to get an information on pointwise properties of~$\psi_n^\eps$
in terms of~$\psi_n^0$ 
(the latter eigenfunction is explicitly known).
For this reason, we look at the partial differential equation
that the difference $\psi_n^\eps - \psi_n^0$ satisfies 
and using the ideas of elliptic regularity theory 
in the spirit of \cite{FK4,K-Tusek_2015} 
and the Sobolev embedding theorem,
we obtain an improved result
\begin{equation}\label{C1-conv} 
  \| \psi_n^\eps - \psi_n^0 \|_{C^{1,\gamma}(\overline{\domain})}
  \leq C_\gamma\eps
\end{equation}
with any $\gamma \in (0,1)$,
where~$C_\gamma$ is an $\eps$-independent constant.

\subsubsection*{Ad \emph{Step~IV}}
Since the extremal points of~$\psi_n^0$ are explicitly known, 
convergence~\eqref{C1-conv} is sufficient to conclude 
with a weaker version of~\eqref{eq:location_max} and~\eqref{eq:location_min}, 
namely with just closures of~$S_0^{(n)}(\delta)$ and $S_{n-1}^{(n)}(\delta)$
instead of the more restricted sets~$\tilde{S}_0^{(n)}$ 
and $\tilde{S}_{n-1}^{(n)}$
on the right-hand side of the formulae.
Here the main idea is to locate the stationary points
with help of~\eqref{C1-conv}.
To exclude the possibility of existence of extremal points
of~$\psi_n^\eps$ in the boundary narrow regions 
$S_0^{(n)}(\delta)$ and $S_{n-1}^{(n)}(\delta)$ as well
(which is particularly crucial for the proof of the hot spots conjecture),
we still try to continue with the ideas of elliptic regularity theory as above,
but now we are only able to show the bound
\begin{equation}\label{C2-bound} 
  \| \psi_n^\eps - \psi_n^0 \|_{C^{2,\gamma}(\overline{\domain})}
  \leq C_\gamma
\end{equation}
with any $\gamma \in (0,1)$,
where~$C_\gamma$ is an $\eps$-independent constant
(possibly different from that of~\eqref{C1-conv}).
Although the right-hand side does not vanish as $\eps \to 0$,
the presence of the H\"older exponent~$\gamma$ in~\eqref{C2-bound}
enables one to eventually exclude the existence of extremal points
in the undesired boundary narrow regions. 

However, as was pointed out by an anonymous referee, 
one can actually proceed more straightforwardly.
In fact,  it is even possible to prove the convergence of eigenfunctions in the ``ideal'' space $C^2(\overline{\domain})$. This can be done using Lemma~\ref{lem:ref} whose variant was proposed by her/him and that we decided to include in the manuscript. Furthermore, as was correctly remarked by the referee, \eqref{C2-bound} may be combined together with the $C^2(\overline{\domain})$-convergence result to obtain
\begin{equation}\label{eq:C2-conv}
 \|\psi_n^\varepsilon-\psi_n^0\|_{C^{2,\gamma}(\overline{\domain})}\leq C_{\gamma,a}\varepsilon^{(1-\gamma)a},
\end{equation}
where $\gamma\in(0,1)$, $a<\frac{1}{2}$, and $C_{\gamma,a}$ is an $\varepsilon$-independent constant, see Remark~\ref{rem:ref}.
We are grateful to the referee for these substantial 
improvements upon our convergence results.

Finally, the claims of Theorem~\ref{Thm.location} about the absence
of local extrema of~$\psi_n^\eps$ in~$S_m^{(n)}(\delta)$
eventually employs the maximum principle, too.

\medskip
\begin{remark}
Since this paper is primarily motivated by the hot spots conjecture,
we focus on the corresponding consequences of~\eqref{C1-conv} and~\eqref{C2-bound}
summarised in Theorem~\ref{Thm.location}. Notice, however, 
that the established convergence results for eigenfunctions 
can be interested in different contexts, too.
For instance, \eqref{C1-conv}~immediately enables one to locate
the nodal lines of eigenfunctions~$\psi_n^\eps$ in terms of 
the nodal lines of~$\psi_n^0$.   
We refer to~\cite{FK4} for this original idea in the context
of Dirichlet strips.
\end{remark}

The rest of the paper consists of four sections 
following the individual steps of the proof 
as indicated above.

%----------------------------%
\section{Fermi coordinates}\label{Sec.geometry}
%----------------------------%
%
The above identification of~$\Omega_\eps$ with $(\domain,G_\eps)$
is justified provided that the exponential map
$\mathscr{L}_\eps : \domain \to \Omega_\eps$ is a diffeomorphism. 
Since~$\Gamma$ is compact, there exists a positive number~$\eps_0$  
depending on the geometry of~$\Gamma$ and~$\mathcal{A}$
such that the required property holds for all $\eps < \eps_0$.
The smallness of~$\eps_0$ is determined by requiring 
that the function~$f_\eps$ from~\eqref{eq:Jacobi} is positive
(to ensure that $\mathscr{L}_\eps : \domain \to \Omega_\eps$ 
is a local diffeomorphism by the inverse function theorem)  
and that~$\mathscr{L}_\eps$ is injective 
(to get a global diffeomorphism).
The former leads us to look at 
the asymptotic behaviour of~$f_\eps$ as $\eps\to 0$.

Firstly, remark that, due to \eqref{eq:K_scaling},
\begin{equation*} 
  \|K_\eps\|_{L^\infty(\domain)}
  = \|K\|_{L^\infty(\Omega_\eps)}
  \leq \|K\|_{L^\infty(\mathcal{A})}
  =: \|K\|_\infty
  \,,
\end{equation*}
where the bound is independent of~$\eps$
(but of course a bound in a fixed neighbourhood of~$\Gamma$
would be sufficient for our purposes). 
Also define $\|\kappa\|_\infty := \|\kappa\|_{L^\infty((0,L))}$. 
With this notation in hand we have
\begin{proposition}
For all $\eps\in\big(0,\sqrt{2/\|K\|_\infty}\big)$ 
(or any $\eps>0$ in the case $K=0$), 
 \begin{equation}\label{eq:f_bound}
    1-C_\eps\leq f_\eps(x)\leq 1+C_\eps \qquad (\forall x\in\domain)
 \end{equation}
 with
 \begin{equation*}
    C_\eps:=\eps\|\kappa\|_\infty+\frac{1}{2}\eps^2\|K\|_\infty\frac{1+\eps\|\kappa\|_\infty}{1-\frac{1}{2}\eps^2\|K\|_\infty}.
 \end{equation*}
\end{proposition}
\begin{proof}
 The problem \eqref{eq:Jacobi} is equivalent to the following integral equation
 \begin{equation} \label{eq:f_int}
  f_\eps(s,t)=1-\eps t\kappa(s)-\eps^2\int_0^t\int_0^\eta (K_\eps f_\eps)(s,\xi)\,\dd\xi\,\dd\eta \,.
 \end{equation}
Since, due to our regularity assumptions, 
$f_\eps$ is as smooth as we need on $\overline{\domain}$, we have
 \begin{equation} \label{eq:f_bound_est}
  \big| \|f_\eps\|_{L^\infty(\domain)}-1 \big|
  \leq \|f_\eps-1\|_{L^\infty(\domain)} 
  \leq \eps\|\kappa\|_\infty
  +\frac{1}{2}\eps^2\|K\|_\infty
  \|f_\eps\|_{L^\infty(\domain)}.
 \end{equation}
If 
 \begin{equation} \label{eq:eps_ass_1}
   \frac{1}{2}\eps^2\|K\|_\infty <1, 
 \end{equation}
this yields
 \begin{equation*} 
   \frac{1-\eps\|\kappa\|_\infty}{1+\frac{1}{2}\eps^2\|K\|_\infty }
   \leq\|f_\eps\|_{L^\infty(\domain)}
   \leq\frac{1+\eps\|\kappa\|_\infty}{1-\frac{1}{2}\eps^2\|K\|_\infty }.
 \end{equation*}
Putting the second inequality into \eqref{eq:f_int} we arrive at \eqref{eq:f_bound}.
\end{proof}
\begin{corollary}
 There exists $\tilde\eps>0$ such that for all $\eps\leq \tilde\eps$, $\mathscr{L}_\eps : \domain \to \Omega_\eps$ 
 is a local diffeomorphism.
\end{corollary}
Note that one can take $\tilde\eps$ arbitrarily close to 
but strictly below the unique solution 
of the equation $C_\eps=1$ 
with respect to all $\eps$ satisfying \eqref{eq:eps_ass_1}. 
The uniqueness follows from the fact that $\eps\mapsto C_\eps$ 
is increasing on $\big(0,\sqrt{2/\|K\|_\infty}\big)$ 
and it maps this interval onto $(0,+\infty)$.

The map~$\mathscr{L}_\eps$ will become 
a global diffeomorphism provided that we further restrict the smallness of~$\eps$ 
by requiring that~$\mathscr{L}_\eps$ is injective (which is always possible because we assume  that~$\Gamma$ is embedded in~$\mathcal{A}$).

\begin{remark}\label{Rem.dependence}
Abandoning the geometrical interpretation of~$\Omega_\eps$
as a tubular neighbourhood of~$\Gamma$ embedded in~$\mathcal{A}$,
one may consider $(\domain,G_\eps)$ as an abstract Riemannian manifold, in which case it is enough to assume that 
$\eps\leq\tilde\eps$. Then $\Omega_\eps$ is just immersed in~$\mathcal{A}$. At the same time, $H_\eps$~is a well defined self-adjoint operator under the mere hypothesis $\eps\leq\tilde\eps$.
\end{remark}

For all $\eps\leq\tilde\eps$ and any non-negative integer~$k$, we have
\begin{align*}
&\|\partial_s^k K_\eps\|_{L^\infty(\domain)}
=\|\partial_s^k K\|_{L^\infty(\domain_\eps)}\leq\|\partial_s^k K\|_{L^\infty(\domain_{\tilde\eps})},\\
&\|\partial_t^k K_\eps\|_{L^\infty(\domain)} 
=\eps^k \, \|\partial_u^k K\|_{L^\infty(\domain_\eps)}\leq \eps^k \, 
\|\partial_u^k K\|_{L^\infty(\domain_{\tilde\eps})},  
\end{align*}
where we introduced the notation $\domain_{\eps}:=(0,L)\times(-\eps,\eps)$, for any $\eps>0$. Remark that the norms on the right-hand sides are $\eps$-independent quantities. With an abuse of notation, we keep the same letter~$K$ 
for $K \circ \mathscr{L}$, where~$\mathscr{L}$ is 
the unscaled exponential map 
$\mathscr{L}(s,u) := \exp_{\Gamma(s)}\big(uN(s)\big)$.
Here and in the sequel,	 we denote by~$(s,u)$ 
a generic variable in~$\domain_\eps$ (respectively, $\domain_{\tilde\eps}$). 

We continue by estimating derivatives of~$f_\eps$.
Since 
 \begin{equation} \label{eq:f_der_t}
  \partial_t f_\eps(s,t)=-\eps\kappa(s)-\eps^2\int_0^t(K_\eps f_\eps)(s,\xi)\,\dd\xi,
 \end{equation}
we obtain
\begin{equation*}
  \|\partial_t f_\eps\|_{L^\infty(\domain)}
  \leq\eps\|\kappa\|_\infty+\eps^2\|K\|_\infty\|f_\eps\|_{L^\infty(\domain)}
  \leq \eps\|\kappa\|_\infty+\eps^2\|K\|_\infty  \frac{1+\eps\|\kappa\|_\infty}{1-\frac{1}{2}\eps^2\|K\|_\infty}=\mathcal{O}(\eps).
\end{equation*}
For the second derivative we get in a straightforward manner
 \begin{equation*}
  \|\partial_t^2 f_\eps\|_{L^\infty(\domain)}
  \leq\eps^2\|K\|_\infty\|f_\eps\|_{L^\infty(\domain)}
  \leq \eps^2\|K\|_\infty \frac{1+\eps\|\kappa\|_\infty}{1-\frac{1}{2}\eps^2\|K\|_\infty}=\mathcal{O}(\eps^2).
 \end{equation*}
The $n$th order derivatives ($n \geq 2$)
with respect to $t$ obey 
 \begin{equation*}
  \partial_t^n f_\eps(s,t)=-\eps^2\sum_{k=0}^{n-2}
  \binom{n-2}{k} \, \eps^k \,
  \big(\partial_u^k K(s,\eps t)\big) \, \partial_t^{n-2-k}f_\eps(s,t).
 \end{equation*}
Therefore, by induction, one has
\begin{equation*}
  \|\partial_t^n f_\eps\|_{L^\infty(\domain)}
  =\mathcal{O}(\eps^n).
\end{equation*}
 
Now for any $n\in\N$,
 \begin{equation*}
  \partial_s^n f_\eps(s,t)=-\eps t \frac{\dd^n \kappa}{\dd s^n}(s)-\eps^2\int_0^t\int_0^\eta \left(\sum_{k=0}^{n}\binom{n}{k}(\partial^k_s K_\eps) \, \partial_s^{n-k}f_\eps\right)(s,\xi)\,\dd\xi \, \dd\eta,
 \end{equation*}
which implies that
 \begin{equation*}
  \left(1-\eps^2\|K\|_\infty\right)\|\partial_s^n f_\eps\|_{L^\infty(\domain)}
  \leq\eps\left\|\frac{\dd^n \kappa}{\dd s^n}\right\|_\infty
  +\frac{1}{2}\eps^2\sum_{k=1}^{n}\binom{n}{k}
  \|\partial_s^k K\|_{L^\infty(\domain_{\tilde\eps})}
  \|\partial_s^{n-k}f_\eps\|_{L^\infty(\domain)}.
 \end{equation*}
Again, by induction, one easily verifies that
 \begin{equation*}
  \|\partial_s^n f_\eps\|_{L^\infty(\domain)}
  =\mathcal{O}(\eps).
 \end{equation*}
In a similar vein one derives that ($m\in\N$)
 \begin{equation*}
   \|\partial_s^n\partial_t^m f_\eps\|_{L^\infty(\domain)}
  =\mathcal{O}(\eps^m)\,.
 \end{equation*}
 
Finally, for any $a\in\R$ and $i\in\{s,t\}$, $\partial_i f_\eps^{a}=a f_\eps^{a-1}\partial_i f_\eps$. Hence we have
 \begin{equation*}
 \|\partial_i f_\eps^{a}\|_{L^\infty(\domain)}
  =\mathcal{O}(\eps) 
 \end{equation*}
and similarly for the higher derivatives.
We summarise our findings in
\begin{proposition}\label{prop:f_asy} 
 For all $a\in\R$,
 \begin{align*}
  &f_\eps^a=1+\mathcal{O}(\eps)\,\text{ uniformly on } \domain,\\
  & \|\partial_s^n\partial_t^m f_\eps^a\|_{L^\infty(\domain)}
  =\begin{cases}\mathcal{O}(\eps)&\text{ for } m=0\text{ and }n\in\N \,, \\
    \mathcal{O}(\eps^m)&\text{ for } m\in\N\text{ and }n\in\N\cup\{0\}\,,
   \end{cases}
 \end{align*}
 as $\eps\to 0$.
\end{proposition}

We will denote the product and norm of $L^2(\domain, f_\eps\,\dd s\,\dd t)$ 
and $L^2(\domain, \dd s \,\dd t)$ by $\langle\cdot,\cdot\rangle_\eps$ 
and $\|\cdot\|_\eps$ and by $\langle\cdot,\cdot\rangle$ 
and $\|\cdot\|$, respectively. 
We will always assume that $\eps\leq\tilde\eps$.
 Then these norms are equivalent, because for all $g\in L^2(\domain,\dd s \, \dd t)$,
 \begin{equation*}
  (1-C_\eps)\|g\|^2\leq\|g\|_\eps^2\leq(1+C_\eps)\|g\|^2,
 \end{equation*}
 where $C_\eps\to0$ as $\eps \to 0	$. 
In particular, for any $a\in\R$, $\|g\|_\eps=\mathcal{O}(\eps^{a})$ if and only if $\|g\|=\mathcal{O}(\eps^{a})$.

%-------------------------------------%
\section{Norm-resolvent convergence}\label{Sec.nrs}
%-------------------------------------%
% 
 In this section we prove Theorem~\ref{theo:res_conv}.

 Since $H_0$ and $H_\eps$ are positive operators, $(H_0+1)$ and $(H_\eps+1)$ are boundedly invertible. For any $F,G\in L^2(\domain,\dd s \, \dd t)$, 
 we can find $\phi\in\dom{H_0}$ and $\psi\in\dom{H_\eps}$ such that
 \begin{equation}\label{eq:op_ran}
 (H_0+1)\phi=F,\qquad (H_\eps +1)\psi=U_{\eps}^{-1}G.
 \end{equation}
  Thus we have
 \begin{multline*}
 \langle F, U_\eps(H_\eps+1)^{-1}U_{\eps}^{-1}G\rangle-\langle F,(H_0 +1)^{-1} G\rangle\\
 =\langle (H_0+1)\phi,U_{\eps}\psi\rangle-\langle\phi,U_{\eps}(H_{\eps}+1)\psi\rangle
 =\langle \phi, H_0 U_{\eps}\psi\rangle-\langle U_{\eps}^{-1}\phi,H_{\eps}\psi\rangle _{\eps}\\
 =h_0(\phi, U_\eps\psi)-h_\eps (U_{\eps}^{-1}\phi,\psi),
 \end{multline*}
 where $h_\eps$ and $h_0$ are quadratic forms of $H_\eps$ and $H_0$, respectively, see \eqref{eq:form_eps} and \eqref{eq:form_0}.
 Hence to show \eqref{eq:res_conv} it is enough to prove that
 \begin{equation}\label{eq:eq_est}
  |h_0(\phi, U_{\eps}\psi)-h_\eps (U_{\eps}^{-1}\phi,\psi)|\leq C\eps\|F\| \|G\|
 \end{equation}
 for all $\phi,\psi\in W^{1,2}(\domain)$ satisfying \eqref{eq:op_ran}.
 
 Multiplying the first equation of \eqref{eq:op_ran} by $\phi$ 
with respect to the dot product of $L^2(\domain,\dd s \, \dd t)$, we obtain
 $$\int_\domain |\phi,_{s}|^2+\eps^{-2}\int_\domain|\phi,_{t}|^2+\int_{\domain}|\phi|^2=\int_{\domain}\bar{\phi}F\leq\|\phi\|\|F\|.$$
 Here and in the sequel, the comma denotes a partial derivative.
 This yields
 \begin{equation} \label{eq:phi_bounds}
  \|\phi\|\leq\|F\|,\quad \|\phi,_s\|\leq\|F\|,\quad \|\phi,_t\|\leq\eps\|F\|.
 \end{equation}
 Similarly from the second equation of \eqref{eq:op_ran} we get
 $$
\int_\domain f_{\eps}^{-1}|\psi,_{s}|^2+\eps^{-2}\int_\domain f_\eps|\psi,_{t}|^2+\int_{\domain} f_\eps|\psi|^2=\int_{\domain}f_{\eps}^{1/2}\bar{\psi}G
\leq\|f_{\eps}^{1/2}\|_{L^\infty(\domain)}\|\psi\|\|G\|.
$$
 Clearly, by \eqref{eq:f_bound},
 \begin{equation*}
  f_{\eps}^{-1}\geq (1+C_\eps)^{-1}, \quad f_{\eps}\geq 1-C_\eps,\quad 
  \|f_{\eps}^{1/2}\|_{L^\infty(\domain)}\leq (1+C_\eps)^{1/2}.
 \end{equation*}
 Thus we have
\begin{equation*}
  \|\psi\|\leq\frac{(1+C_\eps)^{1/2}}{(1-C_\eps)}\|G\|,
  \;\; \|\psi,_s\|\leq\frac{1+C_\eps}{(1-C_\eps)^{1/2}}\|G\|,
  \;\; \|\psi,_t\|\leq\frac{\eps (1+C_\eps)^{1/2}}{(1-C_\eps)}\|G\|.
 \end{equation*}
 We conclude that there exist $C_1,\,\eps_1>0$ such that, for all $\eps<\eps_1$, we have
 \begin{equation} \label{eq:psi_bounds}
  \|\psi\|\leq C_1\|G\|,\quad \|\psi,_s\|\leq C_1\|G\|,\quad \|\psi,_t\|\leq C_1\eps\|G\|.
 \end{equation}

 After a straightforward calculation one gets
 \begin{multline*}
  h(\phi,U_\eps\psi)-h_\eps(U_{\eps}^{-1}\phi,\psi)=\eps^{-2}\int_\domain (f_{\eps}^{1/2}),_t (\bar{\phi},_t \psi+\bar{\phi}\psi,_t)+\int_\domain(f_{\eps}^{1/2}),_s \bar{\phi},_s \psi\\-\int_\domain (f_{\eps}^{-1/2}),_s f_{\eps}^{-1}\bar{\phi}\psi,_s-\int_\domain (f_{\eps}^{-3/2}-f_{\eps}^{1/2})\bar{\phi},_s \psi,_s.
 \end{multline*}
 Since, by Proposition \ref{prop:f_asy},
 \begin{align*}
  (f_{\eps}^{1/2}),_t&=\mathcal{O}(\eps),
  & (f_{\eps}^{1/2}),_s&=\mathcal{O}(\eps), \\
  (f_{\eps}^{-1/2}),_sf_{\eps}^{-1}&=\mathcal{O}(\eps),
  & f_{\eps}^{-3/2}-f_{\eps}^{1/2}&=\mathcal{O}(\eps),
 \end{align*}
 uniformly on $\domain$ as $\eps\to 0$, 
we infer that there exist $C_2, \, \eps_2>0$ such that, 
for all $\eps<\eps_2$, 
 \begin{equation}\label{eq_inter_est}
  |h(\phi,U_\eps\psi)-h_\eps(U_{\eps}^{-1}\phi,\psi)-\eps^{-2}\int_\domain (f_{\eps}^{1/2}),_t (\bar{\phi},_t \psi+\bar{\phi}\psi,_t)|\leq C_2\eps\|F\|\|G\|.
 \end{equation}
 Here we used the Schwarz inequality together with \eqref{eq:phi_bounds} and \eqref{eq:psi_bounds}.
 
 Multiplying the first and the second equation of \eqref{eq:op_ran} by $\kappa t\psi$ and $\kappa t\phi$ with respect to the dot product of 
$L^2(\domain,\dd s \, \dd t)$ and $L^2(\domain,f_\eps\,\dd s\,\dd t)$, 
respectively, we obtain
 \begin{equation*}
 \int_\domain \kappa t \bar{\psi},_s\phi,_s+\int_\domain\frac{\dd\kappa}{\dd s}t\bar{\psi}\phi,_s+\eps^{-2}\int_\domain\kappa t\bar{\psi},_t\phi,_t+\eps^{-2}\int_\domain\kappa\bar{\psi}\phi,_t+\int_\domain\kappa t\bar{\psi}\phi
 =\int_\domain\kappa t\bar{\psi}F
 \end{equation*}
 and
 \begin{multline*}
  \int_\domain \kappa t f_{\eps}^{-1} \bar{\phi},_s\psi,_s+\int_\domain\frac{\dd \kappa}{\dd s}tf_{\eps}^{-1}\bar{\phi}\psi,_s+\eps^{-2}\int_\domain\kappa tf_{\eps}\bar{\phi},_t\psi,_t\\
  +\eps^{-2}\int_\domain\kappa f_{\eps}\bar{\phi}\psi,_t
  +\int_\domain\kappa t f_{\eps}\bar{\phi}\psi=\int_\domain\kappa t f_{\eps}^{1/2}\bar{\phi}G.
 \end{multline*}
 Adding the complex-conjugate of the first equation to the second equation we arrive at
 \begin{multline*}
  \!\!\int_\domain \kappa t (1+f_{\eps}^{-1})\bar{\phi},_s \psi,_s +\int_\domain \frac{\dd\kappa}{\dd s}t(\bar{\phi},_s \psi+f_{\eps}^{-1}\bar{\phi}\psi,_s)+\eps^{-2}\int_\domain \kappa t(\bar{\phi},_t \psi,_t+f_{\eps}\bar{\phi},_t \psi,_t)\\
  +\eps^{-2}\int_{\domain}\kappa (\bar{\phi},_t \psi+\bar{\phi}\psi,_t)+\eps^{-2}\int_\domain \kappa (f_\eps-1)\bar{\phi}\psi,_t+\int_\domain\kappa t(1+f_\eps)\bar{\phi}\psi\\
  =\int_\domain \kappa t \psi\bar{F}+\int_\domain \kappa t f_{\eps}^{1/2}\bar{\phi}G.
 \end{multline*}
 This means that for some $C_3,\,\eps_3>0$ we have 
 \begin{equation*}
  \left|
  \eps^{-2}\int_{\domain}\kappa (\bar{\phi},_t \psi+\bar{\phi}\psi,_t)
  \right|
  \leq C_3\|F\|\|G\|,
 \end{equation*}
 whenever $\eps<\eps_3$.
 
 Putting this estimate together with \eqref{eq_inter_est} and the observation that 
 $$(f_\eps^{1/2})_{,t}=-\frac{1}{2}\eps\kappa+\mathcal{O}(\eps^2)$$ 
 (see \eqref{eq:f_der_t}), we arrive at \eqref{eq:eq_est}.
 This concludes the proof of Theorem~\ref{theo:res_conv}.
\hfill\qed

\begin{remark}
Inspecting the proof above and recalling Remark~\ref{Rem.dependence},
the dependence of the constants~$C$ and~$\eps_0$ 
on the geometry of~$\Gamma$ and~$\mathcal{A}$ in Theorem~\ref{theo:res_conv}
is only through the supremum norms 
$\|\kappa\|_\infty$, $\|\kappa'\|_\infty$,
$\|K\|_\infty$, and 
$\|\partial_s K\|_{L^\infty((0,L)\times(-\tilde\eps,\tilde\eps))}$.
\end{remark}
%

%---------------------------------------%
\section{Convergence of eigenfunctions}\label{Sec.efs}
%---------------------------------------%
%
The eigenpairs of $H_0$ are of the form $(\lambda_n^0+\nu_m,\,u_n^0\otimes\chi_m)$, $n\in\N,\,m\in\N\cup\{0\}$, where
 \begin{equation*}
  \lambda_n^0\:=\left(\frac{(n-1)\pi}{L}\right)^2,
  \quad \nu_m\:=\left(\frac{m\pi}{2\eps}\right)^2,
 \end{equation*}
are eigenvalues of one-dimensional Neumann problems
after a separation of variables
and
 \begin{equation*}
  u_n^0(s) := \begin{cases}
              \frac{1}{\sqrt{L}}&\text{ for }n=1 , \\
              \sqrt{\frac{2}{L}}\cos{\frac{(n-1)\pi s}{L}}
&\text{ for }n \geq 2,
            \end{cases} 
 \;
 \chi_m(t) := \begin{cases}
              \frac{1}{\sqrt{2}}&\text{ for }m=0 , \\
              \cos{\frac{m\pi t}{2}}&\text{ for }m\geq 2\text{ even} , \\
              \sin{\frac{m\pi t}{2}}&\text{ for }m\text{ odd,}
             \end{cases}
 \end{equation*}
are the corresponding normalised eigenfunctions.
Take $N\in\N$. If we rearrange the eigenvalues in non-decreasing order then, for all $\eps$ such that
\begin{equation}\label{rearrange}
  \eps<\frac{L}{2(N-1)},
\end{equation}
the first $N$ eigenpairs are just $(\lambda_n^0,\,\psi_n^0)$, where 
 $$\psi_n^0 := u_n^0\otimes\chi_0=u_n^0\otimes \frac{1}{\sqrt{2}}\qquad (n=1,2,\ldots,N).$$
Note that these eigenvalues are simple and  all $\psi_n^0$ are normalised 
to one in $L^2(\domain,\dd s \, \dd t)$.

\subsection{$L^2$-convergence}
Recall that $(\lambda_n^\eps,\,\psi_n^\eps)$, $n\in\N$, 
stands for the $n$th eigenpair of~$H_\eps$. 
As a consequence of Theorem~\ref{theo:res_conv}, 
$\psi_n^\eps$ may be chosen in such a way 
that $\|U_\eps\psi_n^\eps-\psi_n^0\|=\mathcal{O}(\eps)$,
which in a view of Proposition~\ref{prop:f_asy} implies that 
(\cf~\eqref{L2-conv})
\begin{equation*}
 \|\psi_n^\eps-\psi_n^0\|=\mathcal{O}(\eps).
\end{equation*}
We also have (\cf~\eqref{ev-conv})
\begin{equation} \label{eq:ev_conv}
 |\lambda_n^\eps-\lambda_n^0|=\mathcal{O}(\eps).
\end{equation}
These results are trivial for $n=1$, 
because the first Neumann eigenpair is always 
formed by a zero eigenvalue and a constant eigenfunction.
Therefore, we will omit this case in our considerations below.

\subsection{$W^{1,2}$-convergence}
For any $n=2,3,\ldots,N$, put
\begin{equation*}
 \psi:=\psi_n^\eps-\psi_n^0.
\end{equation*}
Recall that we have 
\begin{equation} \label{eq:ef_conv}
 \|\psi\|=\mathcal{O}(\eps)
\end{equation}
as $\eps\to 0$.
Subtracting the eigenvalue equations for $\psi_n^\eps$ and $\psi_n^0$ and taking into account that $\partial_t \psi_n^0=0$, we infer that
\begin{equation}\label{eq:master}
 -f_{\eps}^{-1}\partial_s f_{\eps}^{-1}\partial_s\psi-\eps^{-2}f_{\eps}^{-1}\partial_t f_{\eps}\partial_t \psi=\alpha,
\end{equation}
where
\begin{equation*}
 \alpha:=\lambda_n^\eps \psi-(\lambda_n^0-\lambda_n^\eps)\psi_n^0+(f_{\eps}^{-2}-1)\partial_{s}^{2}\psi_n^0+f_{\eps}^{-1}(f_\eps^{-1})_{,s}\partial_s\psi_n^0.
\end{equation*}
With the help of \eqref{eq:ev_conv}, \eqref{eq:ef_conv}, and Proposition \ref{prop:f_asy} we deduce that
\begin{equation} \label{eq:alpha_conv}
 \|\alpha\|=\mathcal{O}(\eps)
\end{equation}
as $\eps\to 0$.

Multiplying the both sides of \eqref{eq:master} by $\psi$ with respect to the dot product of $L^2(\domain,f_{\eps}\dd s\,\dd t)$ and employing the Schwarz inequality on the right-hand side of the result, we arrive at
\begin{equation*}
\langle f_{\eps}^{-1}\partial_s\psi,\partial_s\psi\rangle+\eps^{-2}\langle f_{\eps}\partial_t\psi,\partial_t\psi\rangle= \langle\alpha,f_\eps\psi\rangle\leq (1+C_\eps)\|\alpha\|\|\psi\|=\mathcal{O}(\eps^2).
\end{equation*}
This yields
\begin{equation} \label{eq:ef_der_conv}
 \|\partial_s\psi\|=\mathcal{O}(\eps),\quad \|\partial_t\psi\|=\mathcal{O}(\eps^2)
\end{equation}
as $\eps\to 0$.

\subsection{$W^{2,2}$-convergence} \label{sec:W2}
Multiplying the both sides of \eqref{eq:master} by $(-\partial_{s}^{2}\psi)$ with respect to the dot product of $L^2(\domain,f_{\eps}\dd s\,\dd t)$, we obtain
\begin{multline}\label{eq:master_2}
 \!\!\!\!\!\langle f_{\eps}^{-1}\partial_{s}^{2}\psi,\partial_{s}^{2}\psi\rangle+\langle (f_{\eps}^{-1})_{,s}\partial_s\psi,\partial_{s}^{2}\psi\rangle+\eps^{-2}\langle f_{\eps}\partial_s\partial_t\psi,\partial_s\partial_t\psi\rangle+\eps^{-2}\langle (f_\eps)_{,s}\partial_t\psi,\partial_t\partial_s\psi\rangle\\
 =-\langle f_\eps\alpha,\partial_s^2\psi\rangle=\langle(f_\eps)_{,s}\alpha,\partial_s\psi\rangle+\langle f_\eps \alpha_{,s},\partial_s\psi\rangle\\
 \leq \|(f_\eps)_{,s}\|_{L^\infty(\domain)} \|\alpha\|\|\partial_s\psi\|
  +\|f_\eps\|_{L^\infty(\domain)} \|\alpha_{,s}\|\|\partial_s\psi\|.
\end{multline}
On the left-hand side we have integrated twice by parts,
$$\langle\partial_t f_\eps \partial_t\psi,\partial_{s}^{2}\psi\rangle=\langle\partial_s f_\eps \partial_t\psi,\partial_t\partial_{s}\psi\rangle.$$
By the Schwarz and Young inequalities, we have
\begin{align*}
 &|\langle(f_{\eps}^{-1})_{,s}\partial_s\psi,\partial_{s}^{2}\psi\rangle|
 \leq \frac{1}{2}\|(f_\eps^{-1})_{,s}\|_{L^\infty(\domain)}
  (\|\partial_s\psi\|^2+\|\partial_{s}^{2}\psi\|^2),
\\
 &|\langle (f_\eps)_{,s}\partial_t\psi,\partial_t\partial_s\psi\rangle|
  \leq \|(f_\eps)_{,s}\|_{L^\infty(\domain)}^2 \|\partial_t\psi\|^2
  +\frac{1}{4}\|\partial_t\partial_s\psi\|^2.
\end{align*}
Now we may further estimate in \eqref{eq:master_2},
\begin{multline*}
 (1+C_\eps)^{-1}\|\partial_{s}^{2}\psi\|^2
-\frac{1}{2}\|(f_\eps^{-1})_{,s}\|_{L^\infty(\domain)}
\|\partial_{s}^{2}\psi\|^2\\
+\eps^{-2}(1-C_\eps)\|\partial_s\partial_t\psi\|^2
-\frac{1}{4\eps^2}\|\partial_s\partial_t\psi\|^2 \\
 \leq \frac{1}{2}\|(f_\eps^{-1})_{,s}\|_{L^\infty(\domain)}
\|\partial_s\psi\|^2+\eps^{-2}\|(f_\eps)_{,s}\|_{L^\infty(\domain)}^2
\|\partial_t\psi\|^2\\
 +\|(f_\eps)_{,s}\|_{L^\infty(\domain)} \|\alpha\|\|\partial_s\psi\|
+\|f_\eps\|_{L^\infty(\domain)} \|\alpha_{,s}\|\|\partial_s\psi\|.
\end{multline*}

Observe that, due to \eqref{eq:ef_der_conv}, \eqref{eq:ev_conv}, \eqref{eq:ef_conv}, and Proposition \ref{prop:f_asy},
\begin{equation} \label{eq:alpha_der_conv}
 \|\partial_i\alpha\|=\mathcal{O}(\eps)\qquad (i=s,t)
\end{equation}
as $\eps\to 0$.
Taking this together with the previous convergence results into the account, we infer that
$$\frac{1}{2}\|\partial_{s}^{2}\psi\|^2+\frac{1}{2\eps^2}\|\partial_s\partial_t\psi\|^2=\mathcal{O}(\eps^2),$$
from which we   conclude that
\begin{equation*}
 \|\partial_{s}^{2}\psi\|=\mathcal{O}(\eps),\quad \|\partial_s\partial_t\psi\|=\mathcal{O}(\eps^2)
\end{equation*}
as $\eps\to 0$.

Finally, \eqref{eq:master} yields
\begin{equation*}
 -\eps^{-2}\partial_{t}^{2}\psi=\alpha+\eps^{-2}f_{\eps}^{-1}(f_{\eps})_{,t}\partial_t\psi+f_{\eps}^{-1}(f_{\eps}^{-1})_{,s}\partial_s\psi+f_\eps^{-2}\partial_s^2\psi.
 \end{equation*}
Taking the norm of the both sides we arrive at
\begin{equation*}
 \|\partial_{t}^{2}\psi\|=\mathcal{O}(\eps^3).
\end{equation*}

\subsection{$W^{3,2}$-convergence}
Due to \eqref{eq:alpha_conv}, \eqref{eq:alpha_der_conv}, and Proposition \ref{prop:f_asy}, we also have
\begin{equation} \label{eq:alpha_der_conv_2}
 \|f_\eps \alpha\|=\mathcal{O}(\eps),\quad \|\partial_i(f_\eps \alpha)\|=\mathcal{O}(\eps)\quad (i=s,t)
\end{equation}
as $\eps\to 0$. Multiplying \eqref{eq:master} by $f_\eps$ and taking the partial derivative of the result with respect to~$s$, we obtain
\begin{equation}\label{eq:master_3}
 -f_\eps^{-1}\partial_s^3\psi-\eps^{-2}f_\eps\partial_s\partial_t^2\psi=\beta,
\end{equation}
where 
\begin{equation*}
 \beta:=\partial_s(f_\eps \alpha)+2(f_\eps^{-1}),_{s}\partial_s^2\psi+(f_\eps^{-1}),_{ss}\partial_s\psi
 +\eps^{-2}\left((f_\eps),_s\partial_t^2\psi+(f_\eps),_{st}\partial_t\psi+(f_\eps),_t\partial_s\partial_t\psi\right).
\end{equation*}
Using \eqref{eq:alpha_der_conv_2} together with Proposition \ref{prop:f_asy} and the convergence results so far derived for the derivatives of $\psi$, one easily infers that
\begin{equation*}
\|\beta\|=\mathcal{O}(\eps),\quad \|f_\eps^{-1}\beta\|=\mathcal{O}(\eps)
\end{equation*}
as $\eps\to 0$.

Next, multiply \eqref{eq:master_3} by $f_\eps^{-1}$ and take the dot  product on $L^2(\domain,\,\dd s\,\dd t)$ of the both sides of the result with $(-\partial_s^3\psi)$. In this way, we arrive at 
\begin{equation} \label{eq:master_3_mult}
 \langle f_\eps^{-2}\partial_s^3\psi,\partial_s^3\psi\rangle+\eps^{-2}\langle\partial_s\partial_t^2\psi,\partial_s^3\psi\rangle=\langle f_\eps^{-1}\beta,\partial_s^3\psi\rangle.
\end{equation}
In the second term, we may integrate by parts as follows
\begin{equation*}
 \langle\partial_s\partial_t^2\psi,\partial_s^3\psi\rangle=-\langle\partial_s\partial_t\psi,\partial_t\partial_s^3\psi\rangle=\langle\partial_s^2\partial_t\psi,\partial_t\partial_s^2\psi\rangle.
\end{equation*}
Moreover, $(1+C_\eps)^{-2}\|\partial_s^3\psi\|^2\leq\langle f_\eps^{-2}\partial_s^3\psi,\partial_s^3\psi\rangle$.
Therefore, we may estimate in \eqref{eq:master_3_mult} as follows
\begin{equation*}
  (1+C_\eps)^{-2}\|\partial_s^3\psi\|^2+\eps^{-2}\|\partial_s^2\partial_t\psi\|^2\leq \langle f_\eps^{-1}\beta,\partial_s^3\psi\rangle
  \leq \|f_\eps^{-1}\beta\|\|\partial_s^3\psi\|
  \leq \frac{1}{2}\left(\|f_\eps^{-1}\beta\|^2+\|\partial_s^3\psi\|^2 \right).
\end{equation*}
Hence, we have
\begin{equation*}
 \left((1+C_\eps)^{-2}-\frac{1}{2}\right)\|\partial_s^3\psi\|^2+\eps^{-2}\|\partial_s^2\partial_t\psi\|^2\leq \frac{1}{2}\|f_\eps^{-1}\beta\|^2=\mathcal{O}(\eps^2).
\end{equation*}
We conclude that
\begin{equation*}
 \|\partial_s^3\psi\|=\mathcal{O}(\eps),\quad \|\partial_s^2\partial_t\psi\|=\mathcal{O}(\eps^2)
 \end{equation*}
as $\eps\to 0$.

Taking the norm of \eqref{eq:master_3} multiplied by $f_\eps^{-1}$ we immediately obtain
\begin{equation*}
 \eps^{-2}\|\partial_s\partial_t^2\psi\|\leq\|f_\eps^{-2}\partial_s^3\psi\|+\|f_\eps^{-1}\beta\|=\mathcal{O}(\eps).
\end{equation*}
Therefore,
\begin{equation*}
 \|\partial_s\partial_t^2\psi\|=\mathcal{O}(\eps^3)
\end{equation*}
as $\eps\to 0$. Finally, multiplying \eqref{eq:master} by $f_\eps$ and taking the partial derivative of the result with respect to $t$, we obtain
\begin{multline*}
 -(f_\eps^{-1}),_t\partial_s^2\psi-f_\eps^{-1}\partial_t\partial_s^2\psi-(f_\eps^{-1}),_{st}\partial_s\psi-(f_\eps^{-1}),_s\partial_t\partial_s\psi\\
 -\eps^{-2}\left(2(f_\eps),_t\partial_t^2\psi+f_\eps \partial_t^3\psi+(f_\eps)_{,tt}\partial_t\psi\right)=\partial_t(f_\eps\beta).
\end{multline*}
Expressing $\partial_t^3\psi$ from this equation, taking the norm of the both sides of the resulting formula, and using the triangle inequality together with the previous convergence results, we infer that
\begin{equation*}
 \|\partial_t^3\psi\|=\mathcal{O}(\eps^3)
\end{equation*}
as $\eps\to 0$.

\subsection{$C^{1,\gamma}$-convergence}
By the general Sobolev inequality, we have
\begin{corollary}\label{cor:Holder}
For any $\gamma\in(0,1)$,
there exists a positive constant~$C_\gamma$ 
depending on~$\gamma$ and~$L$ only such that
 \begin{equation*}
  \|\psi\|_{C^{1,\gamma}(\overline{\domain})}
  \leq C_\gamma\|\psi\|_{W^{3,2}(\domain)}=\mathcal{O}(\eps).
 \end{equation*}
 In particular,
 \begin{equation}\label{eq:unif_conv}
  \|\psi\|_{L^\infty(\domain)}
  =\mathcal{O}(\eps),\quad 
  \|\partial_s \psi\|_{L^\infty(\domain)}=\mathcal{O}(\eps)
 \end{equation}
 as $\eps\to 0$.
\end{corollary}

\subsection{$W^{4,2}$-bound}
Taking the partial derivative of \eqref{eq:master_3} with respect to $t$ and multiplying the result by $f_\eps^{-1}$  we obtain
\begin{equation}\label{eq:master_4}
 -f_\eps^{-2}\partial_s^3\partial_t\psi-\eps^{-2}\partial_s\partial_t^3\psi=\zeta,
\end{equation}
where
\begin{equation*}
 \zeta:=f_\eps^{-1}\left((\beta)_{,t}+(f_\eps^{-1})_{,t}\partial_s^3\psi+\eps^{-2}(f_\eps)_{,t}\partial_s\partial_t^2\psi\right).
\end{equation*}
Using the previous convergence results we infer that
\begin{equation} \label{eq:gamma_conv}
\|\zeta\|=\mathcal{O}(\eps)
\end{equation}
as $\eps\to 0$.

Next, we multiply both sides of equation~\eqref{eq:master_4} 
by $(-\partial_s^3\partial_t\psi)$ with respect to the dot product on $L^2(\domain,\dd s \, \dd t)$,
$$\langle f_\eps^{-2}\partial_s^3\partial_t\psi,\partial_s^3\partial_t\psi\rangle+\eps^{-2}\langle\partial_s\partial_t^3\psi,\partial_s^3\partial_t\psi\rangle=\langle\zeta,-\partial_s^3\partial_t\psi\rangle. $$
Using the three relations
\begin{align*}
 &(1+C_\eps)^{-2}\|\partial_s^3\partial_t\psi\|^2 \leq \langle f_\eps^{-2}\partial_s^3\partial_t\psi,\partial_s^3\partial_t\psi\rangle , \\
 &\langle\partial_s\partial_t^3\psi,\partial_s^3\partial_t\psi\rangle=-\langle\partial_s^2\partial_t^3\psi,\partial_s^2\partial_t\psi\rangle=\|\partial_s^2\partial_t^2\psi\|^2, \\
 &\langle\zeta,-\partial_s^3\partial_t\psi\rangle=|\langle\zeta,-\partial_s^3\partial_t\psi\rangle|\leq \|\zeta\|^2+\frac{1}{4}\|\partial_s^3\partial_t\psi\|^2,
\end{align*}
we arrive at
\begin{equation*}
\left((1+C_\eps)^{-2}-\frac{1}{4}\right) \|\partial_s^3\partial_t\psi\|^2+\eps^{-2}\|\partial_s^2\partial_t^2\psi\|^2\leq\|\zeta\|^2,
\end{equation*}
which yields
\begin{equation*}
 \|\partial_s^3\partial_t\psi\|=\mathcal{O}(\eps),\quad 
 \|\partial_s^2\partial_t^2\psi\|=\mathcal{O}(\eps^2),
\end{equation*}
as $\eps\to 0$.
If we express $\partial_s\partial_t^3\psi$ from \eqref{eq:master_4} and use the first result together with \eqref{eq:gamma_conv}, we obtain
\begin{equation*}
 \|\partial_s\partial_t^3\psi\|=\mathcal{O}(\eps^3).
\end{equation*}

Now, taking the partial derivative of \eqref{eq:master_3} with respect to $s$, in a similar vein as above, we arrive at
\begin{equation} \label{eq:master_5}
 -f_\eps^{-2}\partial_s^4\psi-\eps^{-2}\partial_s^2\partial_t^2\psi=f_\eps^{-1}\left(\beta_{,s}+(f_\eps^{-1})_{,s}\partial_s^3\psi+\eps^{-2}(f_\eps)_{,s}\partial_s\partial_t^2\psi\right)=:\tilde\zeta,
\end{equation}
where
$$\|\tilde\zeta\|=\mathcal{O}(\eps)$$
as $\eps\to 0$. However, multiplying \eqref{eq:master_5} by $\partial_s^4\psi$ does not make so much use, because there remains a boundary term after integration by parts,
$$
  \langle\partial_s^2\partial_t^2\psi,\partial_s^4\psi\rangle=\|\partial_s^3\partial_t\psi\|^2-\int_{\{0,L\}\times (-1,1)}
  (\partial_s^2\partial_t\psi) \,
  \partial_s^3\partial_t\psi,
$$
which may be controlled using the trace theorem only 
by $W^{1,2}$-norms of $\partial_s^2\partial_t\psi$ 
and $\partial_s^3\partial_t\psi$. Instead of it, one can  estimate $\langle\partial_s^2\partial_t^2\psi,\partial_s^4\psi\rangle$ using the Schwarz and Young inequalities. In this way, we arrive at
$$\|\partial_s^4\psi\|=\mathcal{O}(1).$$
Remark that one can obtain the same result even more directly 
just by taking the $L^2$-norms in~\eqref{eq:master_5}.

Finally, taking the $L^2$-norms 
in the second derivative of \eqref{eq:master} with respect to~$t$, 
we get
$$\|\partial_t^4\psi\|=\mathcal{O}(\eps^3)$$
(for flat strips, \ie~$K=0$, 
we get a better result $\|\partial_t^4\psi\|=\mathcal{O}(\eps^4)$).
We conclude that
\begin{equation*}
 \|D^4\psi\|=\mathcal{O}(1)
\end{equation*}
as $\eps\to 0$.

\subsection{$C^{2,\gamma}$-convergence}
By the general Sobolev inequality, we have

\begin{corollary}\label{cor:Sobolev2}
For any $\gamma\in(0,1)$,
there exists a positive constant~$\tilde{C}_\gamma$ 
depending on~$\gamma$ and~$L$ only such that
 \begin{equation*}
  \|\psi\|_{C^{2,\gamma}(\overline{\domain})}
  \leq \tilde{C}_\gamma\|\psi\|_{W^{4,2}(\domain)}
  =\mathcal{O}(1).
 \end{equation*}
\end{corollary}

The following observation is essentially due to the referee.
\begin{lemma}\label{lem:ref}
 Let $\Omega$ be a bounded open subset of $\R^d$ with Lipschitz boundary.  
 If the family $\{\phi_\varepsilon\}_{\varepsilon>0}$ is a bounded set in $C^{0,\gamma}(\overline\Omega)$ and $\|\phi_\varepsilon\|_{L^2(\Omega)}=\mathcal{O}(\varepsilon)$ as $\varepsilon\to 0$ then 
$$
  \|\phi_\eps\|_{C^0(\overline\Omega)}
  =\mathcal{O}(\varepsilon^b)\;\text{as } \varepsilon\to 0 
  \text{ with } b:=\frac{2}{2+\frac{d}{\gamma}}.
$$ 
\end{lemma}
\begin{proof}
 If this is not the case, then for any $K,\varepsilon>0$ there exists $\delta\in(0,\varepsilon)$ and $x_\delta\in\Omega$ such that $|\phi_\delta(x_\delta)|>K\delta^b$. Since there exists $\delta$-independent constant $C>0$ such that for all $x,y\in\Omega$, $|\phi_\delta(x)-\phi_\delta(y)|\leq C|x-y|^\gamma$, it is not difficult to see that for all 
$$
  y\in M_\delta:=\left\{y\in\Omega : \ 
  C|x_\delta-y|^\gamma<\frac{K}{2}\delta^b\right\},
$$
 $|\phi_\delta(y)|>\frac{K}{2}\delta^b$. This together with $\|\phi_\varepsilon\|_{L^2(\Omega)}=\mathcal{O}(\varepsilon)$ implies 
\begin{equation} \label{eq:vol_bound}
 \tilde{K}\delta^2>\|\phi_\delta\|_{L^2(\Omega)}^2\geq\|\phi_\delta\|_{L^2(M_\delta)}^2>\Big(\frac{K}{2}\Big)^2\delta^{2b}|M_\delta|,
\end{equation} 
 where the first inequality holds true with some fixed $\tilde{K}>0$ and for all $\delta$ small enough. Now, we have
 $$|M_\delta|=\left|B\Big(x_\delta,\Big(\frac{K}{2C}\delta^b\Big)^{\frac{1}{\gamma}}\Big)\cap\Omega\right|>LK^\frac{d}{\gamma}\delta^{\frac{b d}{\gamma}}$$
 with some $\delta$-independent $L>0$ whose existence is guaranteed by the assumption that $\Omega$ is bounded and with Lipschitz boundary.
 Therefore, the both sides of \eqref{eq:vol_bound} decay as $\delta^2$ when $\delta\to 0$. But  $K$ might have been chosen arbitrarily large.  Hence, for a sufficiently small $\varepsilon$, 
 we arrive at a contradiction with \eqref{eq:vol_bound}.
\end{proof}
%%%%
Since by Corollary \ref{cor:Sobolev2}, the H\"older coefficients of all the second derivatives of $\psi$ are bounded, and in Section \ref{sec:W2} we proved that $\|D^2\psi\|=\mathcal{O}(\varepsilon)$, we arrive at
the following consequence of Lemma~\ref{lem:ref}:
%%%%
\begin{proposition} \label{prop:C2_conv}
 For any $a<\frac{1}{2}$, there exists a positive $\varepsilon$-independent constant $K_a$ such that
 \begin{equation*}
 \|\psi\|_{C^2(\overline\domain)}=\|\psi_n^\varepsilon-\psi_n^0\|_{C^2(\overline\domain)}\leq K_a\varepsilon^a.
 \end{equation*}
\end{proposition}
%%%%
\begin{remark} \label{rem:ref}
Putting Proposition \ref{prop:C2_conv} and Corollary \ref{cor:Sobolev2} together we obtain
\begin{equation*}
 \|\psi\|_{C^{2,\tilde\gamma}(\overline\domain)}=\mathcal{O}\big(\varepsilon^{(1-\tilde\gamma)a}\big)
\end{equation*}
with any $\tilde\gamma\in(0,1)$ and $a<\frac{1}{2}$. This can be easily seen after applying the following estimate on the second derivatives of $\psi$,
\begin{equation*}
\frac{|\phi(x)-\phi(y)|}{|x-y|^{\tilde\gamma}}\leq\left(\frac{|\phi(x)-\phi(y)|}{|x-y|^{\gamma}}\right)^{\frac{\tilde\gamma}{\gamma}}(2\|\phi\|_{C^{0}(\overline\domain)})^{1-\frac{\tilde\gamma}{\gamma}},
\end{equation*}
where $0<\tilde\gamma<\gamma\in(0,1)$ and $\phi\in C^{0,\gamma}(\overline\domain)$.
\end{remark}

%-----------------------------------------%
\section{Convergence of extremal points}\label{Sec.critical}
%-----------------------------------------%
%
In this section we particularly prove Theorem~\ref{Thm.location}. 

Since, under the smallness assumption~\eqref{rearrange},
we explicitly have
$$
  \psi_n^0(s,t)=\sqrt{\frac{2}{L}} \, \cos{\frac{(n-1)\pi s}{L}},
$$
the global extrema of $\psi_n^0$ on $\overline\domain$ are attained at 
$\{s_m^{(n)}\}_{m=0}^{n-1}\times I$, 
where the points $s_m^{(n)}$ were introduced in \eqref{eq:stac_points} 
and now we abbreviate $I:=(-1,1)$.
More concretely, along $\{0\}\times I$ there are always global maxima; 
along $\{L\}\times I$ there are global maxima if $n$ is odd 
and global minima if $n$ is even; 
along $\{s_m^{(n)}\}\times I,\, m=1,2,\ldots, n-2$, 
there are local minima if $m$ is odd 
and local maxima if $m$ is even. 
All these local extrema are global at same moment and non-strict.

\begin{lemma} \label{lem:extrema}
For any  $N \in \{2,3,\ldots\}$ and $\delta>0$, 
there exists a positive constant~$\eps_1$ 
depending on~$\delta$, $N$, $\Gamma$, and~$\mathcal{A}$ such that, 
for all $\eps < \eps_1$ and any $n \in \{2,\dots,N\}$,
\begin{alignat}{2}
\sup_{S_m^{(n)}(\delta)}\psi_n^\eps
&>\sup_{\domain\setminus\bigcup\limits_{\substack{j=0\\ j\text{ even}}}^{n-1} S_j^{(n)}(\delta)}\psi_n^\eps \qquad 
&&\text{ if }m\text{ is even},  \label{eq:max}
\\ \rule{0ex}{5ex} 
\inf_{S_m^{(n)}(\delta)}\psi_n^\eps
&<\inf_{\domain\setminus\bigcup\limits_{\substack{j=0\\ j\text{ odd}}}^{n-1} S_j^{(n)}(\delta)}\psi_n^\eps \qquad 
&&\text{ if }m\text{ is odd}.  \label{eq:min}
\end{alignat}
\end{lemma}
\begin{proof}
Take 
$$
  \delta\in\left(0,\frac{L}{2(n-1)}\right).
$$ 
(This is not restrictive at all, because for larger values of $\delta$ we may use $\eps_1$ that corresponds to any $\delta$ from the interval.) Let $m$ be even. Then, due to \eqref{eq:unif_conv}, there exists a positive constant $M$  such that for all $t\in\bar{I}$ and all $\eps$ sufficiently small,
\begin{equation*}
 \psi_n^\eps(s_m^{(n)},t)>\sqrt{\frac{2}{L}}-M\eps.
\end{equation*}
On the other hand, for all $\eps$ small enough, 
\begin{equation*}
 \psi_n^\eps<\sqrt{\frac{2}{L}}\cos\frac{(n-1)\pi\delta}{L}+M\eps 
\qquad \text{on} \qquad \domain\setminus\bigcup\limits_{\substack{j=0\\ j\text{ even}}}^{n-1} S_j^{(n)}(\delta).
\end{equation*}
Therefore,  \eqref{eq:max} is valid, whenever
\begin{equation*}
 2M\eps< \sqrt{\frac{2}{L}}\Big(1-\cos\frac{\pi\delta}{L}\Big).
\end{equation*}
Inequality \eqref{eq:min} may be verified in a similar manner.
\end{proof}

\begin{lemma} \label{lem:no_stac_1}
For any  $N \in \{2,3,\ldots\}$ and $\delta>0$, 
there exists a positive constant~$\eps_2$ 
depending on~$\delta$, $N$, $\Gamma$, and~$\mathcal{A}$ such that, 
for all $\eps < \eps_2$ and any $n \in \{2,\dots,N\}$, 
$$
\mbox{$\psi_n^\eps$ has no stationary point in 
$\displaystyle \domain\setminus\bigcup_{m=0}^{n-1}S_m^{(n)}(\delta)$.}
$$

\end{lemma}
\begin{proof}
Take 
$$
  \delta\in\left(0,\frac{L}{2(n-1)}\right).
$$ 
(For the same reason as in the proof of Lemma \ref{lem:extrema}, this is actually no restriction.)
Using \eqref{eq:unif_conv} we infer that, 
for all 
$
  (s,t)\in \domain\setminus\bigcup_{m=0}^{n-1}S_m^{(n)}(\delta),
$
$$
  |\partial_s\psi_n^\eps(s,t)|
  \geq \frac{\sqrt{2}(n-1)\pi}{L^{3/2}} \sin{\frac{(n-1)\pi\delta}{L}}
  -M\eps
  \geq\frac{2^{3/2}\pi(n-1)^2}{L^{5/2}}\delta-M\eps,
$$
where $M$ is a positive constant 
and $\eps$ is small enough. 
We see that for all~$\eps$ sufficiently small,
\begin{equation} \label{eq:der_est}
 |\partial_s\psi_n^\eps|>0
  \qquad \text{on} \qquad
  \domain\setminus \bigcup_{m=0}^{n-1}S_m^{(n)}(\delta).
\end{equation}
\end{proof}
\begin{lemma} \label{lem:no_stac_2}
For any  $N \in \{2,3,\ldots\}$ and $\delta\in\left(0, \frac{L}{4(N-1)}\right)$
there exists a positive constant~$\eps_3$  
depending on~$N$, $\Gamma$, and~$\mathcal{A}$ such that, 
for all $\eps < \eps_3$ and any $n \in \{2,\dots,N\}$, 
$$
\mbox{$\psi_n^\eps$ has no stationary point in 
$\displaystyle S_0^{(n)}(\delta)\cup S_{n-1}^{(n)}(\delta)$.}
$$
\end{lemma}
\begin{proof}
 We will prove the claim on $S_0^{(n)}(\delta)$, where we have
 \begin{equation*}
 \partial^2_s\psi_n^0(s,t)=-\sqrt{\frac{2}{L}}\left(\frac{(n-1)\pi}{L}\right)^2 \cos\frac{(n-1)\pi s}{L}<-\sqrt{\frac{1}{L}}\left(\frac{(n-1)\pi}{L}\right)^2.
 \end{equation*}
 Pick $a\in(0,\frac{1}{2})$. Then, by Proposition \ref{prop:C2_conv},  
\begin{align*}
  K_a\varepsilon^a
  &\geq\sup_{\domain}
  |\partial_s^2\psi_n^\varepsilon-\partial_s^2\psi_n^0|
  \\
  &\geq\sup_{S_0^{(n)}(\delta)}(\partial_s^2\psi_n^\varepsilon-\partial_s^2\psi_n^0)
  \\
  &>\sup_{S_0^{(n)}(\delta)}
  \partial_s^2\psi_n^\varepsilon+\sqrt{\frac{1}{L}}\left(\frac{(n-1)\pi}{L}\right)^2.
\end{align*}
 Therefore, for all $\varepsilon$ sufficiently small, $\partial_s^2\psi_n^\varepsilon<0$ in $S_0^{(n)}(\delta)$. Now assume there is a stationary point, say $(s_0,t_0)$, of $\psi_n^\eps$ in $S_0^{(n)}(\delta)$. Then
\begin{equation*}
\int_{0}^{s_0}\partial_{s}^{2}\psi_n^\eps(s,t_0)\dd s=\partial_s\psi_n^\eps(s_0,t_0)-\partial_s\psi_n^\eps(0+,t_0)=0.
\end{equation*}
But the integral on the left-hand side is negative -- a contradiction.
\end{proof}

\subsection*{Proof of Theorem~\ref{Thm.location}}	
Taking $\delta$ from Lemma \ref{lem:no_stac_2} for $\delta$ 
in Lemmata~\ref{lem:extrema} and~\ref{lem:no_stac_1}, 
and identifying $\Omega_\eps$ 
with the Riemannian manifold~$(\domain,G_\eps)$; 
we arrive at  \eqref{eq:location_max}, \eqref{eq:location_min}, and the first statement of the second part of Theorem~\ref{Thm.location} (about stationary points).

Now, for all $s\in(s_m^{(n)}-\delta,s_m^{(n)}+\delta)$ with $\delta<L/(4(n-1))$ and $m=1,2,\ldots,n-2$, we have
\begin{equation*}
  \psi_n^0(s_m^{(n)},t)=(-1)^m \sqrt{\frac{2}{L}},
  \qquad |\psi_n^0(s_m^{(n)},t)-\psi_n^0(s,t)|
  \leq\sqrt{\frac{2}{L}}\left(1-\frac{\sqrt{2}}{2}\right).
 \end{equation*}
Therefore, due to \eqref{eq:unif_conv}, 
$(-1)^m\psi_n^\eps>0$ on $S_m^{(n)}(\delta)$, 
whenever $\eps$ is sufficiently small. 
The eigenfunction equation then yields 
$$
  (-1)^m H_\eps \psi_n^\eps>0
  \qquad \text{on} \qquad
  S_m^{(n)}(\delta).
$$
By the maximum principle, if $m$ is odd then there are no maxima of 
$\psi_n^\eps$ in $S_m^{(n)}(\delta)$  and if~$m$ is even 
then there are no minima of $\psi_n^\eps$ in $S_m^{(n)}(\delta)$. 
This proves the last statement of Theorem~\ref{Thm.location}.% 
\hfill\qed%

\begin{remark}\label{Rem.dependence.bis} 
The more general approach suggested in Remark~\ref{Rem.dependence}
enables one to restate the main results of this paper 
in a more uniform way. 
For instance, inspecting the proof of Theorem~\ref{Thm.location},
we may restate Corollary~\ref{Corol.location} as follows.
Given positive constants~$\tilde{C}$ and~$\tilde{L}$,
let us consider the class of Riemannian manifolds $(\domain,G_\eps)$
satisfying the bounds
\begin{equation}\label{eq:strip_ass}
  L > \tilde{L}
  \;\; \& \;\;
  \sup_{(s,u)\in(0,L)\times(-\tilde{\eps},\tilde{\eps})}
  \left(\sum_{i=0}^{3}\left(
  |\kappa^{(i)}(s)| 
  + |\partial_s^i K(s,u)|\right) 
  + |\partial_u K(s,u)|
  \right)
  < \tilde{C}.
\end{equation}
Then there exists a positive constant~$\eps_0$ 
depending only on~$\tilde{C}$, $\tilde{\eps}$ and~$\tilde{L}$
such that, for any $\eps \leq \eps_0$, 
the second Neumann eigenfunction~$\psi_2^\eps$ of~$H_\eps$
achieves its global maximum (respectively, minimum) 
on the boundary points $\tilde{S}_0^{(2)}$
(respectively, $\tilde{S}_1^{(2)}$).
\end{remark}
%

%-----------------%
% ACKNOWLEDGMENTS %
%-----------------%

\section*{Acknowledgments}
The research of D.K.\ was partially supported by FCT (Portugal)
through project PTDC/MAT-CAL/4334/2014 and by the Czech Science
Foundation (GA\v{C}R) grant No. 18-08835S.
M.T.\ was supported by the Czech Science
Foundation (GA\v{C}R) grant No. 17-01706S.

We are indebted to the anonymous referee 
for providing insightful comments 
which led to improvement of some of the presented results.

%--------------%
% BIBLIOGRAPHY %
%--------------%

%\bibliographystyle{amsplain}

%


\begin{thebibliography}{10}

\bibitem{Atar-Burdzy_2004}
R.~Atar and K.~Burdzy, \emph{On {N}eumann eigenfunctions in lip domains}, J.
  Amer. Math. Soc. \textbf{17} (2004), 243--265.

\bibitem{Banuelos-Burdzy_1999}
R.~Ba{\~n}uelos and K.~Burdzy, \emph{On the ``hot spots'' conjecture of
  {J}.~{R}auch}, J. Funct. Anal. \textbf{164} (1999), 1--33.

\bibitem{Bass-Burdzy_2000}
R.~F. Bass and K.~Burdzy, \emph{Fiber {B}rownian motion and the ``hot spots''
  problem}, Duke Math. J. \textbf{105} (2000), 25--58.

\bibitem{Burdzy_2005}
K.~Burdzy, \emph{The hot spots problem in planar domains with one hole}, Duke.
  Math. \textbf{129} (2005), 481--502.

\bibitem{Burdzy-Werner_1999}
K.~Burdzy and W.~Werner, \emph{A counterexample to the ``hot spots''
  conjecture}, Ann. Math. \textbf{149} (1999), 309--317.

\bibitem{Freitas_2002}
P.~Freitas, \emph{Closed nodal lines and interior hot spots of the second
  eigenfunction of the {L}aplacian on surfaces}, Indiana Univ. Math.~J.
  \textbf{51} (2002), 305--316.

\bibitem{FK4}
P.~Freitas and D.~Krej\v{c}i\v{r}\'{\i}k, \emph{Location of the nodal set for
  thin curved tubes}, Indiana Univ. Math. J. \textbf{57} (2008), no.~1,
  343--376.

\bibitem{Gray}
A.~Gray, \emph{Tubes}, Addison-Wesley Publishing Company, New York, 1990.

\bibitem{Ha_1964}
P.~Hartman, \emph{Geodesic Parallel Coordinates in the Large}, Amer. J. Math. 
  \textbf{86} (1964), 705--727.
  
\bibitem{Jerison-Nadirashvili_2000}
D.~Jerison and N.~Nadirashvili, \emph{The ``hot spots'' conjecture for domains
  with two axes of symmetry}, J. Amer. Math. Soc. \textbf{13} (2000), 741--772.

\bibitem{Kato}
T.~Kato, \emph{Perturbation theory for linear operators}, Springer-Verlag,
  Berlin, 1966.

\bibitem{Kawohl_1985}
B.~Kawohl, \emph{Rearrangements and convexity of level sets in {PDE}}, Lecture
  Notes in Mathematics, vol. 1150, Springer-Verlag, Berlin, 1985.

\bibitem{K5}
D.~Krej\v{c}i\v{r}\'{\i}k, \emph{{Spectrum of the Laplacian in a narrow curved
  strip with combined Dirichlet and Neumann boundary conditions}}, ESAIM:
  Control, Optimisation and Calculus of Variations \textbf{15} (2009),
  555--568.

\bibitem{K-Tusek_2015}
D.~Krej\v{c}i\v{r}\'{\i}k and M.~Tu{\v{s}}ek, \emph{Nodal sets of thin curved
  layers}, J. Differential Equations \textbf{258} (2015), 281--301.

\bibitem{Miyamoto_2009}
Y.~Miyamoto, \emph{The ``hot spots'' conjecture for a certain class of planar
  convex domains}, J. Math. Phys. \textbf{50} (2009), 103530.

\bibitem{Miyamoto_2013}
Y.~Miyamoto, \emph{A planar convex domain with many isolated ``hot spots'' on the
  boundary}, Jpn. J. Ind. Appl. Math. \textbf{30} (2013), 145--164.

\bibitem{Pascu_2002}
M.~N. Pascu, \emph{Scaling coupling of reflecting {B}rownian motions and the
  hot spots problem}, Trans. Amer. Math. Soc. \textbf{354} (2002), 4681--4702.

\bibitem{Post}
O.~Post, \emph{Spectral analysis on graph-like spaces}, Lecture Notes in
  Mathematics, Springer-Verlag, Berlin Heidelberg, 2012.

\bibitem{Rauch_1975}
J.~Rauch, \emph{Five problems: an introduction to the qualitative theory of
  partial differential equations}, Partial differential equations and related
  topics (Program, Tulane Univ., New Orleans, La., 1974) (Berlin), Lecture
  Notes in Mathematics, vol. 446, Springer, 1975, pp.~617--636.

\bibitem{Saito_2001}
Y.~Sait{\={o}}, \emph{Convergence of the {N}eumann {L}aplacian on shrinking
  domains}, Analysis \textbf{21} (2001), 171--204.

\bibitem{Schatzman_1996}
M.~Schatzman, \emph{On the eigenvalues of the {L}aplace operator on a thin set
  with {N}eumann boundary conditions}, Applicable Anal. \textbf{61} (1996),
  293--306.

\bibitem{Siudeja_2015}
B.~Siudeja, \emph{Hot spots conjecture for a class of acute triangles}, Math.
  Z. \textbf{280} (2015), 783--806.

\end{thebibliography}
\end{document}